\documentclass[11pt,reqno]{amsart}
\usepackage{amssymb, amscd, stmaryrd}
\usepackage{algorithm}
\usepackage{algpseudocode}
\usepackage{enumerate}
\usepackage{cases}
\usepackage{amsfonts}
\usepackage[dvips,dvipdf]{graphicx}
\usepackage[usenames,dvipsnames]{color}
\usepackage{color}
\usepackage{color, colortbl}
\usepackage{multirow}
\usepackage[table]{xcolor}
\usepackage{comment}
\usepackage{subfigure}
\usepackage{epstopdf}
\usepackage{tikz}
\usepackage{chngcntr}
\usepackage[stable]{footmisc}
\usepackage{booktabs}

\usepackage{geometry}
\makeatletter
\newcommand\figcaption{\def\@captype{figure}\caption}
\newcommand\tabcaption{\def\@captype{table}\caption}
\makeatother

\counterwithin*{equation}{section}

\newcommand{\be}{\begin{eqnarray}}
\newcommand{\ee}{\end{eqnarray}}
\newcommand{\ben}{\begin{eqnarray*}}
\newcommand{\een}{\end{eqnarray*}}

\def\be{\begin{equation}}
\def\ee{\end{equation}}
\def\bes{\begin{equation*}}
\def\ees{\end{equation*}}
\def\beq{\begin{eqnarray}}
\def\eeq{\end{eqnarray}}
\def\beqs{\begin{eqnarray*}}
\def\eeqs{\end{eqnarray*}}
\def\bal{\begin{aligned}}
\def\eal{\end{aligned}}
\def\bsqs{\begin{subequations}}
\def\esqs{\end{subequations}}

\newtheorem{theorem}{Theorem}[section]

\newtheorem{lemma}[theorem]{Lemma}

\theoremstyle{definition}

\theoremstyle{definition}
\theoremstyle{remark}
\newtheorem{remark}[theorem]{Remark}
\newtheorem{example}[theorem]{Example}

\newtheorem{assumption}[theorem]{Assumption}

\makeatletter
\newcommand{\definetitlefootnote}[1]{%
  \newcommand\addtitlefootnote{%
    \makebox[0pt][l]{$^{\bigstar}$}%
    \footnote{\protect\@titlefootnotetext}
  }%
  \newcommand\@titlefootnotetext{\spaceskip=\z@skip $^{\bigstar}$#1}%
}
\makeatother

\title[A finite element method for Singularly Perturbed problems]{A Mixed Finite Element Method for Singularly Perturbed Fourth Oder Convection - Reaction - Diffusion Problems on Shishkin Mesh}

\author[C. D. Wickramasinghe]{ Charuka D. Wickramasinghe}
\address{Department of Mathematics, University of Kentucky, Lexington, KY 40508, USA}
\email{cdwi263@uky.edu}

\begin{document}


\begin{abstract}
This paper introduces an approach to decoupling singularly perturbed boundary value problems for fourth-order ordinary differential equations that feature a small positive parameter $\epsilon$ multiplying the highest derivative. We specifically examine Lidstone boundary conditions and demonstrate how to break down fourth-order differential equations into a system of second-order problems, with one lacking the parameter and the other featuring $\epsilon$ multiplying the highest derivative. To solve this system, we propose a mixed finite element algorithm and incorporate the Shishkin mesh scheme to capture the solution near boundary layers. Our solver is both direct and of high accuracy, with computation time that scales linearly with the number of grid points. We present numerical results to validate the theoretical results and the accuracy of our method. 

\end{abstract}

\maketitle


\section{Introduction}

In this paper, we consider the following stationary state fourth order singularly perturbed differential equation with Lidstone boundary conditions. 

\begin{equation}\label{e1}
\left.\begin{aligned}
 L_\epsilon &= -\varepsilon u^{iv}(x) - a(x)u^{'''}(x)+ b(x)u^{''}(x) = -f(x) \\
  u(0) &= 0,  \quad  u(1) =0, \quad 
u^{''}(0) = 0, \quad  u^{''}(1)=0
\end{aligned}\right\} \quad 
\end{equation}

The function $f(x)$, $a(x)$, and $b(x)$ are all smooth and satisfy $a(x)\geq \alpha > 0$ and $b(x)\geq \beta > 0$. The parameter $\epsilon$ is assumed to be a small positive value, such that $0< \epsilon \leq 1$.  While it may be easy to analytically solve this problem in some cases, finding the solution $u$ with analytical techniques can be difficult or even impossible for a general function $f$. The well-posedness of problem (\ref{e1}) has been discussed in more detail in sources such as \cite{Ehme} and \cite{sun}.

The convection-diffusion-reaction equation is used in three processes \cite{Makungu}: convection, which involves the movement of materials from one region to another; diffusion, which involves the movement of materials from an area of high concentration to an area of low concentration; and reaction, which involves decay, adsorption, and the reaction of substances with other components.
Singularly perturbed problems have many applications in engineering and applied mathematics, including chemicals and nuclear engineering \cite{Alhumaizi}, linearized Navier-Stokes equation at high Reynolds number (\cite{Amando}, \cite{Kreiss}), elasticity, aerodynamics, oceanography, meteorology, modeling of semiconductor devices \cite{Polak}, control theory, and oil extraction from underground reservoirs \cite{Ewing}, and in many other fields. However, solving these problems numerically presents major computational difficulties due to boundary layers, where the solution changes rapidly. The study of second-order singularly perturbed differential equations are quite large, as seen in the extensive literature (\cite{Doolan}, \cite{FARRELL},\cite{Kadalbajoo},\cite{Mohan},\cite{Natesan},\cite{ OMalley},\cite{Hirsch},\cite{Ewing},\cite{Kopteva}) and their references. However, very few studies have addressed singularly perturbed fourth-order boundary value problems in the literature.

The following paragraph presents a brief overview of analytical and numerical methods used for solving singularly perturbed fourth-order differential equations. In \cite{Vrabel}, Vrabel, R studied the existence of solutions for fourth-order nonlinear equations with the Lidstone boundary conditions. In \cite{El-Zahar}, Essam R introduced the Differential Transform Method (DTM) as an alternative to existing methods for solving higher-order singularly perturbed boundary value problems (SPBVPs). In \cite{VR}, Vrabel, R provided a detailed analysis of the boundary layer phenomenon subject to the Lidstone boundary conditions by analyzing the integral equation associated with the SPBVPs. In \cite{Sun1}, Sun, G, and Martin S developed Piecewise polynomial Galerkin finite-element methods on a Shishkin mesh, which achieved almost optimal convergence results in various norms. In 1991, Ross and Stynes studied a fourth-order problem and generated an approximate solution using patched basis functions. The method is uniformly first-order convergent in the $H^{1}$ [0,1] norm \cite{Roos}. In \cite{Ramanujam}, Shanthi, V., and N. Ramanujam transformed the SVBVP into a system of weakly coupled system of two second-order ODEs and then used the fitted operator method (FOM), fitted mesh method (FMM), and boundary value technique (BVT) to approximate the solution.

The problem being studied involves rapidly changing solutions in very thin regions near the boundary. Traditional numerical methods often fail to accurately capture these changes, which can result in errors across the entire domain. To address this issue, various methods such as Bakhavalov and Gartland meshes have been developed (\cite{Miller}, \cite{Martin}, \cite{Bakhvalov}). In this study, we analyze a standard finite element method combined with the Shishkin mesh, which is a type of local refinement strategy introduced by a Russian mathematician Grigorii Ivanovich Shishkin in 1988 \cite{Natalia}. Finite element methods on Shishkin meshes in 1D were first studied 1995 by Sun and Stynes \cite{sun}, the analysis for second order problems was also published in the two books (\cite{Miller}, \cite{Hans}) from 1996. The goal is to propose a new mixed finite element algorithm that is reliable, effective, and easy to implement, and can be used to solve higher order singularly perturbed differential equations. As far as the authors are aware, there is currently no finite element algorithm available to solve (\ref{e1}) using their decoupled formulation under the Shishkin mesh. The main advantage of decoupling system is that it reduces both memory space and time requirements.

\begin{remark}
For simplicity, the current paper focuses on analyzing a 1-dimensional problem where the coefficients $a(x)$ and $b(x)$ are both constant. However, the analysis could be extended to 2 dimensions and variable coefficients, although this may present some challenges. Additionally, the problem could be expanded to include non-homogeneous boundary conditions through a simple linear transformation.
\end{remark}

The rest of the article is organized as follows: In section 2 we introduce the decouple formulation of (1.1). In section 3 we present Shishkin mesh method and the finite element algorithm. We also present error estimate results of the decouple formulations. In section 4 we present numerical results to validate out theoretical results and conclusion is the section 5. Throughout the following text, the generic positive constants $C$, $a$ and $b$ may take different values in different formulas but is always independent of the mesh and the small positive parameter $\epsilon$.

\section{The Decouple Formulation}\label{sec2}

Hereafter, we will consider the following problem by setting $a(x)=a=$ constant and $b(x)=b=$ constant in equation (\ref{e1}).

\begin{equation}\label{e2}
\left.\begin{aligned}
 L_\epsilon &= -\varepsilon u^{iv}(x) - au^{'''}(x)+ bu^{''}(x) = -f(x) \\
  u(0) &= 0,  \quad  u(1) =0, \quad 
u^{''}(0) = 0, \quad  u^{''}(1)=0
\end{aligned}\right\} \quad 
\end{equation}

The function $f(x)$, $a(x)$, and $b(x)$ are assumed to be sufficiently smooth  for $0\leq x\leq 1$ where, 

\begin{equation}\label{bc}
\left.\begin{aligned}
a & \geq \alpha > 0 \\
b & \geq \beta > 0\\
a + \frac{1}{2}b^{'} & > c>0 \hspace{0.5cm} \mbox{for all }  x \in [0,1]
\end{aligned}\right\} \quad 
\end{equation}

Under the conditions in (\ref{bc}) the problem \ref{e2} is well posed \cite{sun}. Let $(\cdot,\cdot )$ denote the usual $L^2(0,1)$ inner product. We define the bilinear form of the equation (\ref{e2}) as follows:

\begin{equation}\label{bili}
A_{\varepsilon}(u,v) = (-\varepsilon u^{''}, v^{''}) + (au^{''}, v^{'}) - (bu^{'}, v^{'}) = -(f,v) \hspace{0.3cm} \mbox{for all} \hspace{0.3cm}  u,v \in H^{2}_0(0,1)
\end{equation}
The weighted energy norm is given by $$ |||v|||=\{ \varepsilon |v|^{2}_2 + ||v||^{2}_1 \}^{\frac{1}{2}} \hspace{0.3cm} \mbox{for all}   \hspace{0.3cm} v \in H^{2}_0(0,1)$$

\begin{lemma}\label{l2}
Assume (\ref{bc}) holds. Then there  exist positive constants $C_1$, and $C_2$ such that for all $u,v \in H^{2}_0(0,1)$, 

\begin{equation}\label{le1}
|A_{\varepsilon}(u,v)| \leq C_1 \varepsilon^{-{\frac{1}{2}}}|||u|||\cdot|||v|||
\end{equation}

and

\begin{equation}\label{le2}
C_2|||u|||^2 \leq |A_{\varepsilon}(u,u)| 
\end{equation}

\end{lemma}

\begin{proof}:  Based on Cauchy-Schwarz inequality it is easy to see that

\begin{equation}
|A_{\varepsilon}(u,v)| \leq \begin{cases}
C_1 |||u|||\cdot||v||_2\\
C_1 ||u||_2\cdot|||v||| 
\end{cases}
\end{equation}
Since $0 < \varepsilon \leq 1$ (\ref{le1}) follows immediately. The proof of (\ref{le2}) can be shown by following the proof of the lemma (2.1) in \cite{Sun1}.
\end{proof}

The weak formulation of (\ref{e2}) is to find $ u\in H^{2}_0(0,1) $ such that
\begin{equation}
A_{\varepsilon}(u,v) = (-f, v), \hspace{0.3cm} \mbox{for all}   \hspace{0.3cm} v \in H^{2}_0(0,1)
\end{equation}

By the Lax-Milgram lemma, (2.5) has a unique solution $u$  in $H^{2}_0(0,1)$.

We are now able to present our decoupled formulation. It is widely recognized that numerical solutions of higher order problems, such as (\ref{e2}), are significantly more challenging than those of lower order problems. To address this issue, we decouple (\ref{e2}) into a system of lower order differential equations, as follows:

\begin{equation}\label{e3}
\left.\begin{aligned}
-w^{''}(x) &=  f(x)  \hspace{0.5cm}\mbox{for} \hspace{0.5cm}  x \in (0,1) \\
w(0) & = 0,  \quad  w(1) = 0
\end{aligned}\right\} \quad
\end{equation}

\begin{equation}\label{e4}
\left.\begin{aligned}
 -\epsilon u^{''}(x) - au^{'}(x)+ bu(x) & = w(x) \\
  u(0) &= 0,  \quad  u(1) =0
\end{aligned}\right\} \quad 
\end{equation}

The equation represented by Equation (\ref{e3}) is a standard Poison equation, which has the same source term $f(x)$ as Equation (\ref{e2}). Assuming that $f(x)$ belongs to $L^2 (0,1)$ and the given boundary conditions for Equation (\ref{e3}) are met, the problem defined by Equation (\ref{e3}) is well-posed, according to \cite{Ciarlet}. These kind of finite element decouple formulations can be found from (\cite{dilhara1}),  (\cite{dilhara2}) for some particular problems in 2D.

Equation (\ref{e4}) is a second-order problem that involves a convection-reaction-diffusion process with a singular perturbation. The source term for this equation is represented by $w(x)$, which is the solution to the problem defined by Equation (\ref{e3}). Under following  assumptions, 

\begin{equation}\label{bc2}
\left.\begin{aligned}
a & \geq \beta > 0\\
b & \geq 0 \\
b + \frac{a^{'}}{2} & > 0 \hspace{0.5cm} \mbox{for all }  x \in [0,1]
\end{aligned}\right\} \quad 
\end{equation}

the problem defined by Equation (\ref{e4}) is also well-posed, as stated in \cite{Roos}.

In order to establish a connection between the solution $u$ obtained from Equations (\ref{e3}) and (\ref{e4}) and the fourth-order problem represented by Equation (\ref{e2}), we present the following lemma. Let us define $H^{m}(0,1)$ as the Sobolev space comprising functions whose $i^{th}$ derivative, $0 \leq i \leq m$, is square-integrable.

\begin{lemma}\label{l1}
The solution $u \in H^{4} (0,1)$ obtained through (\ref{e3}) 
 and (\ref{e4}) satisfies the following fourth order differential equation.
 
\begin{equation}\label{e5}
\left.\begin{aligned}
-\varepsilon u^{iv}(x) - au^{'''}(x)+ bu^{''}(x) &= -f(x) \\
  u(0) = 0,  \quad  u(1) =0, \quad 
u^{''}(0) = 0, \quad  u^{''}(1)&=0
\end{aligned}\right\} \quad 
\end{equation}
\end{lemma}

\begin{proof} : We first apply the differential operator $L=d^{2}/dx^{2}$ to both sides of Equation (\ref{e4}), which gives us:

\begin{equation}\label{e6}
-\varepsilon u^{iv}(x) - au^{'''}(x)+ bu^{''}(x)= w^{''}(x)
\end{equation}

As we know from Equation (\ref{e3}), $ -w^{''}(x) = f(x) $. Therefore, the equation:

$$ -\varepsilon u^{iv}(x) - au^{'''}(x)+ bu^{''}(x)= -f(x) $$

holds.

To verify the boundary conditions, we apply the differential operator $L$ to the boundary conditions $u(0)=0$ and $u(1)=0$ in Equation (\ref{e4}). This yields the boundary conditions $u^{''}(0)=0$ and $u^{''}(1)=0$. Thus, the conclusion of Lemma (\ref{l1}) holds.
\end{proof}

\section{The Finite Element Method on Shishkin Mesh}

In this section, we present a linear finite element method to solve the singularly perturbed boundary value problem (\ref{e2}) based on the results obtained in the previous section. We utilize the Shishkin mesh to present error estimates for the singularly perturbed convection reaction diffusion problem (\ref{e4}), and we adopt the same definitions and notations as in \cite{sun}.

\subsection{Layer Adapted Shishkin Mesh } 

Given an even positive integer N, the Shishkin mesh $X_s^N$ is constructed by dividing the interval [0,1] into two subintervals. We shall consider a mesh  $ X_s^N:  0 = x_{0} <  x_{1} <  x_{2} <  \cdot \cdot\cdot <x_{n-1}<x_{n} = 1  $
that is equidistant in $[\tau, 1] $ but graded in $[0, \tau]$, where we choose the transition point $\tau $ as Shishkin does:

$$\tau =min \{1/2, (s+1)\alpha^{-1}\varepsilon N\}$$
which depends on $\varepsilon$ and N where, s is the order of the highest derivative.

\begin{figure}[!h]
\centering
\includegraphics[width=12cm]{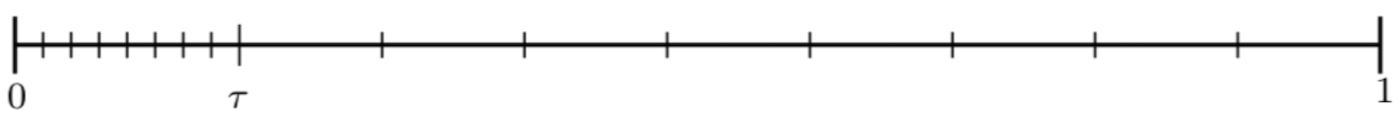}
\begin{tiny}
\caption{1D Shishkin mesh with transition point $\tau $.}\label{shikinmesh}
\end{tiny}
\end{figure}

\begin{assumption}
 In this paper we will assume that
\begin{equation}\label{e17}
    \epsilon \leq CN^{-1}
\end{equation}
as is generally the case for discretizations of convection-dominated problems.
\end{assumption}

\subsection{The Finite Element Algorithm}

Let V be a Hilbert space with norm  $\|\cdot\|_{v}$  (but we shall often omit the subscript v to simplify the notation) and scalar product  . In the discretization of second-order differential equations with domain $I$, one generally chooses $V$ as a subset of the Sobolev space  $H^{1} (I)$.

Let $ I=[0,L] $ be an interval and let the n+1   node points  
$ \{x_{i}\}_{i=0}^{n} $ define a partition. $ I: 0 = x_{0} <  x_{1} <  x_{2} <  \cdot \cdot\cdot <x_{n-1}<x_{n} = L  $ of  $I$ into n sub intervals of length $ h_{i} = x_{i}-x_{i-1}$ for $i=1, \cdot, \cdot, \cdot, n$, and $H=max_{i}h_{i}$.  On the mesh $I$ we define the space $V_{n}   \subseteq V $ of continuous piecewise linear functions by $ V_{n}=\{ v: v \in  C^{0}(I),  v|_{I_i} \in  P_I (I_i)\} $ where $C^{0} (I) $ denotes the space of continuous functions on $I$, and $P_I (I_i) $ denotes the space of linear functions on $I_i$.

\begin{lemma}\label{l2}
Let, $A^{n}_{\varepsilon}(u,v)$ be the discrete bilinear form of $A_{\varepsilon}(u,v)$ in the equation (\ref{bili}). There exists a positive constant $h_0$ (independent of $\varepsilon$) such that for $H\leq h_0$, we have
\begin{equation}\label{e17a}
    C_{1}|||v|||^{2} \leq  A^{n}_{\varepsilon}(v,v) \hspace{0.3cm} \mbox{for all} \hspace{0.3cm} v \in H^2_0
\end{equation}
\end{lemma}

\begin{proof}:  The proof can be done by using  the argument of Lemma 4.1 of \cite{Sun1}, with an additional integration by parts. 
\end{proof}

Then the Galerkin discretization of problem (\ref{e2}) is to find  $u_n \in V_n $such that

\begin{equation}\label{e17ab}
    A^{n}_{\varepsilon}(u_n,v) = (f,v)  \hspace{0.3cm} \mbox{for all} \hspace{0.3cm} v \in V_n
\end{equation}

We will use the following notations:

\begin{gather*}
a(u, v)= \int_{0}^{1} u^{'}(x)v^{'}(x)\, dx \\
(f, v)= \int_{0}^{1} f(x)v(x)\, dx\\
b(u, v)= \varepsilon a(u,v) + (u^{'}, v) + (u, v) \hspace{0.5cm} \mbox{for all}  \hspace{0.5cm}  u,v \in V :=H_0^1(0,1)
\end{gather*}

The variational formulation of (\ref{e3}) is to  find $w  \in V :=H_0^1(0,1) $ such that
\begin{equation}\label{e12}
    a(w,v)= (f,v)  \hspace{0.3cm} \mbox{for all} \hspace{0.3cm} v \in V
\end{equation}

The variational formulation of (\ref{e4}) is to find  $u  \in V :=H_0^1(0,1) $ such that

\begin{equation}\label{e13}
    b(u,v)= (f,v)  \hspace{0.3cm} \mbox{for all} \hspace{0.3cm} v \in V
\end{equation}

The following algorithm summarizes the basic steps in computing the finite element solution $w_{n}$ for the two point boundary value problem (\ref{e3}).

\begin{algorithm}[H]\label{femal}
\caption{}\label{alg:cap1}
\begin{algorithmic}

\item \textbf{Step 1:} Create a mesh $ X_s^N:  0 = x_{0} <  x_{1} <  x_{2} <  \cdot \cdot\cdot <x_{n-1}<x_{n} = 1  $  and define the corresponding space of continuous piecewise linear functions $V_{n,0} = \{v \in V_{n}: v(0)=v(1)=0\}$ with hat basis functions.

\begin{equation*}
\varphi_{j}(x_i)=\begin{cases}
1& \text{if $i=j$ },\\
0& \text{if $i\neq j$ } \hspace{0.5cm}\mbox{for } \hspace{0.5cm} i,j=0,1,2, \cdot \cdot\cdot, n.
\end{cases}
\end{equation*}

\item \textbf{Step 2:} Compute the $(n-1) \times(n-1) $ matrix A and the    $(n-1) \times 1 $ vector b,with entries 

\begin{equation}\label{algoe1}
    A_{i,j}= \int_{0}^{1} \varphi^{'}_{j} \varphi^{'}_{i} \, dx  \hspace{1cm} b_{i}= \int_{0}^{1} f\varphi_{i}\, dx
\end{equation}

\item \textbf{Step 3: } Solve the linear system

\begin{equation}\label{algoe2}
    A\xi = b
\end{equation}

\item \textbf{Step 4: } Set
\begin{equation}\label{algoe3}
    w_n= \sum_{j=1}^{n-1}\xi_{j}\varphi_{j}
\end{equation}

\end{algorithmic}
\end{algorithm}

\begin{figure}[!t]
\centering
\includegraphics[width=8cm]{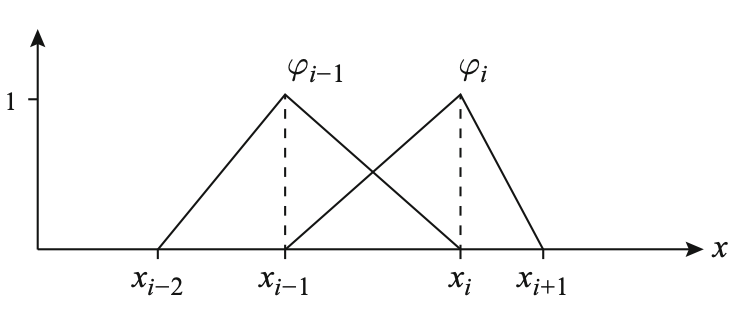}
\begin{tiny}
\caption{The linear hat basis functions in 1D }\label{hatbasis}
\end{tiny}
\end{figure}

In the same way, we can also define the finite element algorithm for $u_n$ in the two-point boundary value problem (\ref{e4}). With these algorithms established, we can now introduce our main mixed finite element algorithm for solving the fourth-order singularly perturbed convection reaction diffusion equation (\ref{e5}). To obtain the finite element solution of the singularly perturbed boundary value problem (\ref{e5}), we utilize the decomposition method described in the previous section.

\begin{algorithm}[H]
\caption{}\label{alg:cap}
\begin{algorithmic}
\item Let $k$ be the order of the interpolation polynomial. For, any $f \in  H^{-1} (0,1) $and $k\geq 1$, we consider the following steps: 
\item \textbf{Step 1:} Find $w_{n}  \in  V_n^k $ in the weak formulation of the Poisson equation (\ref{e3}) on a Shishkin mesh $X_s^N$ 

\begin{equation}\label{e15}
    a(w_n,v)= (f,v)  \hspace{0.3cm} \mbox{for all} \hspace{0.3cm} v \in V
\end{equation}

\item \textbf{Step 2:} Set, 

\begin{equation}\label{e15a}
    -w_n= f_\epsilon 
\end{equation}

where $f_\epsilon$ is the source terms of the equation singularly purtubed second order differential equation (\ref{e4}).

\item \textbf{Step 3:} Then find $ u_n \in V_n^k$ in the weak formulation of the equation  (\ref{e4}) on the Shishkin mesh $X_s^N$ for sufficiently large N (an even positive integer) independent of $ \epsilon$,

\begin{equation}\label{e16}
    b(u_n,v)= (f_\epsilon,v)  \hspace{0.3cm} \mbox{for all} \hspace{0.3cm} v \in V
\end{equation}

\end{algorithmic}
\end{algorithm}

The discretized linear systems corresponding to the stiffness matrix S, convection matrix C, and mass matrix M are obtained as shown below. Equation (\ref{e24}) represents the discretized linear system of (\ref{e15}), while equation (\ref{e25}) corresponds to the discretized linear system of (\ref{e16}).

\begin{equation}\label{e24}
    \left(\frac{\epsilon}{h}S \right ) W= F_1  
\end{equation}

\begin{equation}\label{e25}
    \left(\frac{\epsilon}{h}S - C + \frac{h}{6}M \right ) U = F_2 
\end{equation}

where $U, F \in \mathbb{R}^{n-1}$ and $S, C, M \in \mathbb{R}^{{n-1}\times {n-1}}$ with:

$$ 
W:= \begin{bmatrix}  w_1 \\ w_2 \\ \vdots \\ w_{n-1} \end{bmatrix} \;, \hspace{0.3cm}
F_1:=\begin{bmatrix} (f,\phi_1) \\ (f,\phi_2) \\ \vdots \\ (f,\phi_{n-1}) \end{bmatrix} \;, \hspace{0.3cm}
U:= \begin{bmatrix} u_1 \\ u_2 \\ \vdots \\ u_{n-1} \end{bmatrix} \;, \hspace{0.3cm}
F_2:= \begin{bmatrix} (f_\epsilon,\phi_1)  \\ (f_\epsilon,\phi_2)  \\ \vdots \\(f_\epsilon,\phi_{n-1})  \end{bmatrix} \;
$$ 

\[
S:=\begin{bmatrix}
2 & -1 \\
-1 & 2 & -1\\
&   \ddots &  \ddots &  \ddots \\
& & -1 & 2 &-1 \\
& & & -1 & 2 \\
\end{bmatrix}
, \hspace{0.3cm}
C:=\frac{1}{2}\begin{bmatrix}
0 & 1 \\
-1 & 0 & 1\\
&   \ddots &  \ddots &  \ddots \\
& & -1 & 0 & 1 \\
& & & -1 & 0 \\
\end{bmatrix} 
\]

\[
M:=\begin{bmatrix}
2 & 1 \\
1 & 4 & 1\\
&   \ddots &  \ddots &  \ddots \\
& & 1 & 4 & 1 \\
& & & 1 & 2 \\
\end{bmatrix} 
\]

\subsection{The Error Estimates} 

In this section, we present maximum norm error estimate result for our model  problem using linear finite elements.
$$ L_\epsilon = -\varepsilon u^{iv}(x) - au^{'''}(x)+ bu^{''}(x) = -f(x) 
$$
$$ u(0) = 0,  \quad  u(1) =0, \quad 
u^{''}(0) = 0, \quad  u^{''}(1)=0 $$

As a special case to our model problem we present maximum norm error estimate result with $\varepsilon=1$ for linear finite elements.

\begin{theorem}
Let u be the solution of problem (\ref{e2}). Let, $u_n$ be the solution of (\ref{e17ab}) on the Shishkin mesh $X_s^N$ and $k$ be the order of the interpolation polynomial. Then for N sufficiently large (independently of ,$\varepsilon $) we have

\begin{equation}\label{er1}
   \|u-u_n\|_{\infty} \leq C(N^{-1}lnN)^{min(2, k+1)}
\end{equation}
\end{theorem}

\begin{proof} : 
    The result follows from the argument of Corollary 5.1 in (\cite{sun}).
\end{proof}

\begin{remark}
The proposed decouple formulation is also valid for the special case of $\varepsilon =1 $. In this case there is no boundary layer and thus the Shishkin mesh can be replaced by an uniformly refined mesh (equidistant). 
\end{remark}

\begin{lemma}\label{l5}
Suppose that we use a sufficiently accurate quadrature rule, namely, that $k+1\geq 2$. Let u be the exact solution to the equation (\ref{e2}) with $\varepsilon =1 $.  and  $u_{n}$ be the finite element approximation to the weak formulation of the  equation (\ref{e4}) with $\varepsilon =1 $.  on a uniform mesh. Then we have

\begin{equation}\label{ee9}
    \|u-u_n\|_{\infty} \leq CN^{-2}
\end{equation}

\begin{proof} : 
    On a uniformly refined mesh it is well known that one has 
 \begin{equation}\label{ee10}
    |||u-u_n||| \leq CN^{-2}
\end{equation}

It is easy to see that

 \begin{equation}\label{ee11}
    \|u-u_n\|_{\infty} \leq  \|u-u_n\|_{1} \leq  |||u-u_n||| 
\end{equation}
Combining the equations (\ref{ee10}) and (\ref{ee11}) the conclusion holds.
\end{proof}
\end{lemma}

\section{Numerical Results}

In this section we present a few numerical experiments to illustrate the computational method discussed in this paper. The numerical experiments are performed on a laptop computer with MATLAB R2022a in MacBook Air with M1 chip. The linear finite elements are used to solve our model problem. We use the following numerical convergence rate to validate the theoretical convergence rates:

$$ R=\frac{ ln\|u-u_n\|_{\infty} - ln\|u-u_{n-1}\|_{\infty}}{ ln2}$$

where $u_n $ is the finite element approximation after n mesh refinements and $u$ is the exact solution at the same mesh level as the finite element approximation $u_n $  is calculated. We use the following model problem to validate theoretical results over the following three examples.

\begin{equation}\label{e22}
    -\epsilon u^{iv}(x) - u^{'''}(x)+ u^{''}(x)= -f(x)  \hspace{0.5cm} for \hspace{0.5cm} x \in (0,1)
\end{equation}

\begin{equation}\label{e22b}
    u(0)=0 , u(1)=0, u^{''}(0)=0, u^{''}(1)=0
\end{equation}
 where $f(x)$  is chosen so that the exact solution is 

\begin{equation}\label{e23}
    u(x)=c_{1}e^{r_1x}+c_{2}e^{r_2x} -\frac{x^2+x+1}{2}-\epsilon
\end{equation}
where, 
$$c_1=-c_2+(\frac{1}{2}+\epsilon)$$
$$c_2=\frac{e^{r_1}(\frac{1}{2}+\epsilon)-(\frac{3}{2}+\epsilon)}{e^{r_1}-e^{r_2}}$$
$$r_1=\frac{2}{1+\sqrt{(1+4\epsilon)}}$$
$$r_2=\frac{2}{1-\sqrt{(1+4\epsilon)}}$$.

\begin{example} The purpose of this example is to demonstrate that the solution obtained through our decoupled formulation converges to the exact solution. We compare the finite element solution to equation (\ref{e22}) with its exact solution (\ref{e23}) using both a uniformly refined mesh and a Shishkin mesh. We present the maximum norm error $||u-u_n||_{\infty}$ for different values of $\varepsilon$ and different mesh sizes $N$. Table 1 summarizes the maximum norm error for $\varepsilon= 10^{-10}, 10^{-8 },$ and $10^{-6}$ for both uniformly refined meshes and Shishkin meshes. It is well known that singularly perturbed problems have boundary layers, and uniformly refined meshes do not capture the solution near boundary layers well. This phenomenon is evident from the Table 1 values for uniformly refined meshes. To address this issue, a Shishkin mesh is introduced to capture the solution. The values in Table 1 demonstrate that our decoupled formulation solution converges to the true solution under Shishkin meshes.

Table 2 presents the maximum norm error $||u-u_n||_{\infty}$ for $\varepsilon= 10^{-4}, 10^{-2},$ and $1$ for both uniformly refined meshes and Shishkin meshes. Although a Shishkin mesh is not required for $\varepsilon= 1,$ we included the error results for completeness. By comparing the results in Table 1 and Table 2, we can observe that as we increase the value of $\varepsilon$ from $ \varepsilon = 10^{-10}$ to $ \varepsilon = 1,$ the solution under uniformly refined meshes gradually converges to the true solution. This observation is strongly in agreement with existing theoretical results.

\begin{remark}
  All of the $N$ and $\varepsilon$ values in Table 1 satisfy Assumption 1. However, for some values of $N$ and $\varepsilon$ in Table 2, we have $\varepsilon \geq CN^{-1}$.
    
\end{remark}

Figures 3 and 5 illustrate the numerical solutions and the exact solution under uniformly refined meshes for different values of $N$ and $\varepsilon$. Figures 4 and 6 show the numerical solutions and the exact solution under corresponding values of $N$ and $\varepsilon$ as in Figures 3 and 5. When the singularly perturbed effect is high, it can be observed from Figures 3 and 5 that the numerical solution does not converge to the true solution under a uniform mesh due to the boundary layer occurring at the point $x=0$. However, from Figures 4 and 6, it can be seen that our algorithm converges to the true solution under a Shishkin mesh.

\begin{table}[!h] 
\centering
\def\arraystretch{1.2}
\caption{$ \|u-u_n\|_{\infty} $ \mbox{ for} $\varepsilon= 10^{-10}, 10^{-8 } \mbox{ and }  10^{-6}  \mbox{ for uniform and Shishkin meshes} $.}
\tabcolsep=10pt
\vspace{-0.5\baselineskip}
\begin{tabular}[c]{|c|c|c|c|c|c|c|}
\hline
& \multicolumn{2}{|c |}{$\epsilon=10^{-10}$}& \multicolumn{2}{|c |}{$\epsilon=10^{-8}$} & \multicolumn{2}{| c|}{$\epsilon=10^{-6}$}\\
\hline
$N$\multirow{12}{*}{} & {Uniform} & {Shishkin} & {Uniform} & {Shishkin} &{Uniform} &{Shishkin} \\
\cline{2-7}
\hline
 4 & 0.0354 &0.0558 & 0.0354 & 0.0559  & 0.0354 & 0.0559 \\
\hline
8  & 0.0297 & 0.0152 & 0.0297 & 0.0153& 0.0297 & 0.0153\\
\cline{1-7}
\hline
16  & 0.0293 &0.0039 & 0.0293 & 0.0040 & 0.0293 & 0.0040\\
\cline{1-7}
\hline
32  &0.0296 &9.8165e-04 & 0.0296 & 0.0010 & 0.0296 & 0.0010\\
\cline{1-7}
\hline
64  & 0.0299 & 2.3380e-04& 0.0299 & 2.5559e-04 & 0.0301 & 2.5535e-04\\
\cline{1-7}
\hline
128  &  0.0300 & 5.2171e-05& 0.0300 & 6.4006e-05 & 0.0311 &6.3812e-05\\
\cline{1-7}
\hline
256  & 0.0301 & 1.5676e-05 & 0.0301 & 1.5991e-05 & 0.0342 & 1.5830e-05\\
\cline{1-7}
\hline
512  &0.0302 & 5.3830e-06& 0.0303& 4.8273e-06 & 0.0429 & 4.7611e-06\\
\cline{1-7}
\hline
1024  & 0.0302 &  1.8953e-06 & 0.0309 & 1.4789e-06 & 0.0512 & 1.4328e-06\\
\cline{1-7}
\hline
2048  &  0.0302 & 7.0754e-07& 0.0329 & 4.4459e-07 & 0.0514 & 4.2862e-07\\
\cline{1-7}
\hline
4096  &  0.0303 &  2.8955e-07 & 0.0395 & 1.3130e-07 & 0.0510 & 1.2852e-07\\
\cline{1-7}
\hline
8192 &  0.0307 & 4.0176e-08 & 0.0500 & 3.8094e-08 & 0.0501 & 3.8150e-08\\
\cline{1-7}
\hline
\end{tabular}\label{t1}
\end{table}

\begin{figure}[!t]
\centering
\subfigure[]{\includegraphics[width=0.32\textwidth]{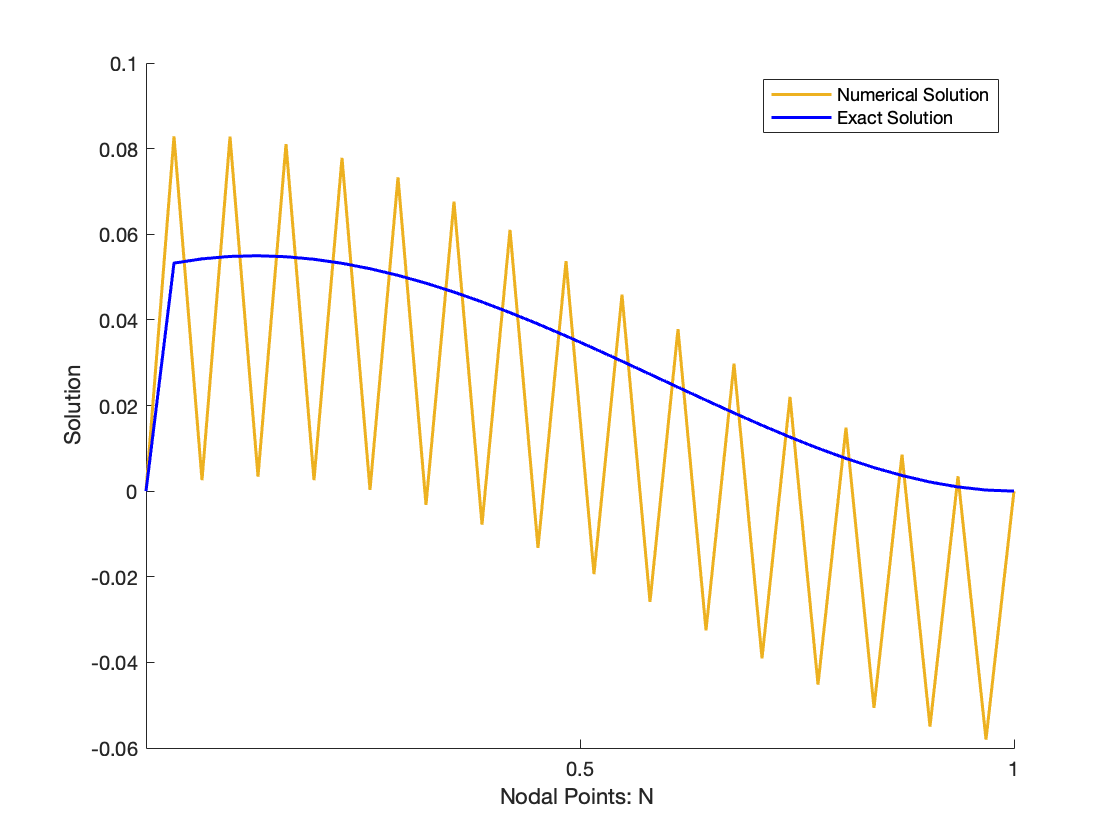}}
\subfigure[]{\includegraphics[width=0.32\textwidth]{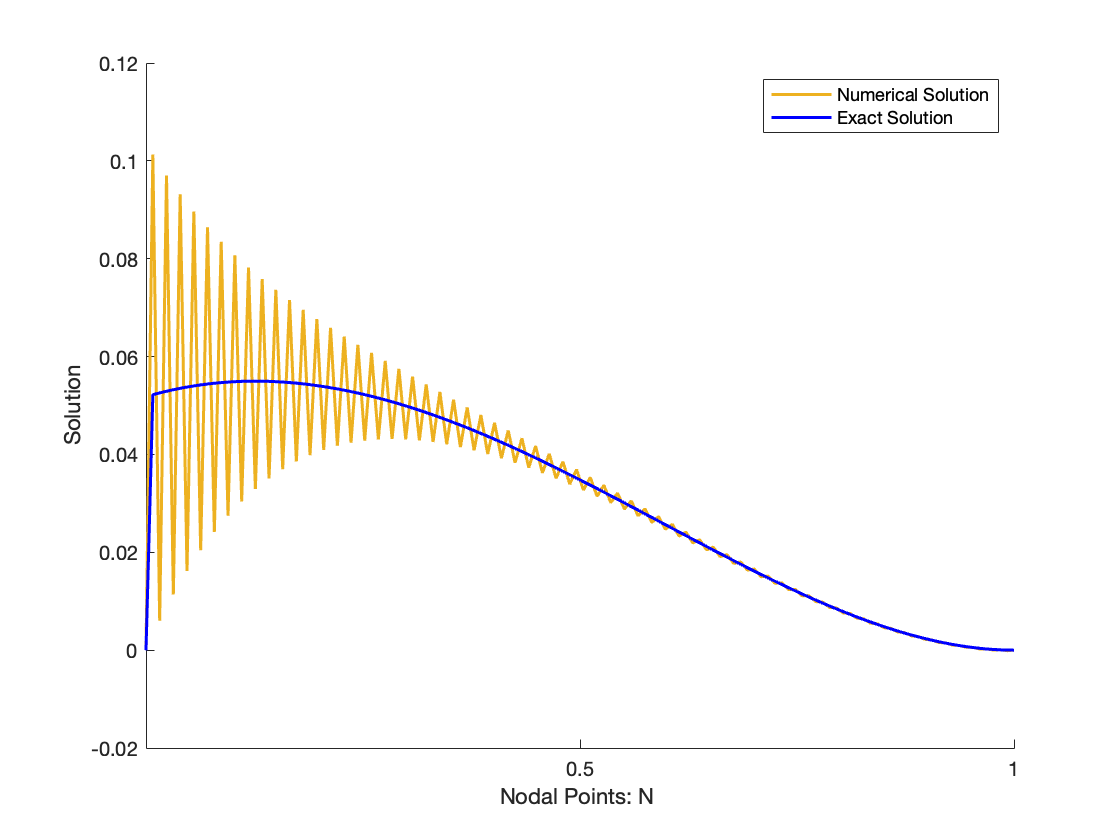}}
\subfigure[]{\includegraphics[width=0.32\textwidth]{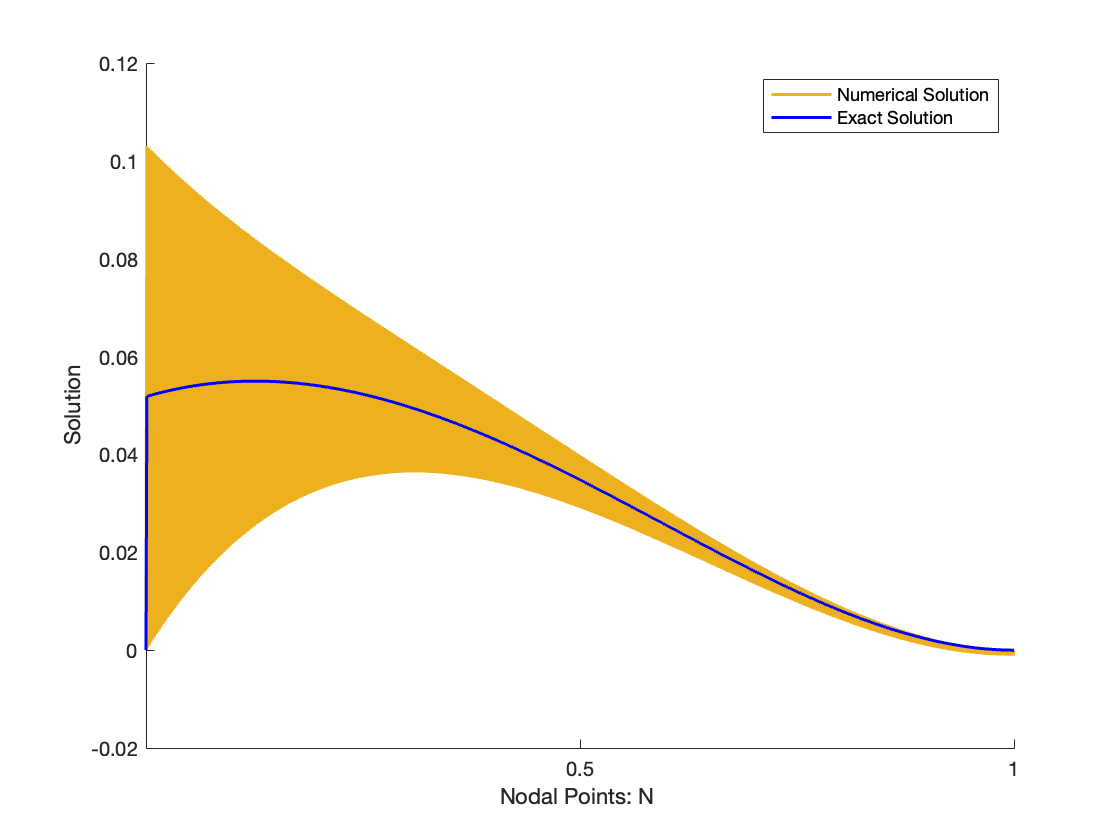}}
\caption{Uniform mesh: (a) $N= 32$, $\varepsilon=10^{-8}$; (b) $N= 128$, $\varepsilon=10^{-4}$; (c) $N= 1024$, $\varepsilon=10^{-6}$.}\label{comp1}
\end{figure}

\begin{figure}[!t]
\centering
\subfigure[]{\includegraphics[width=0.32\textwidth]{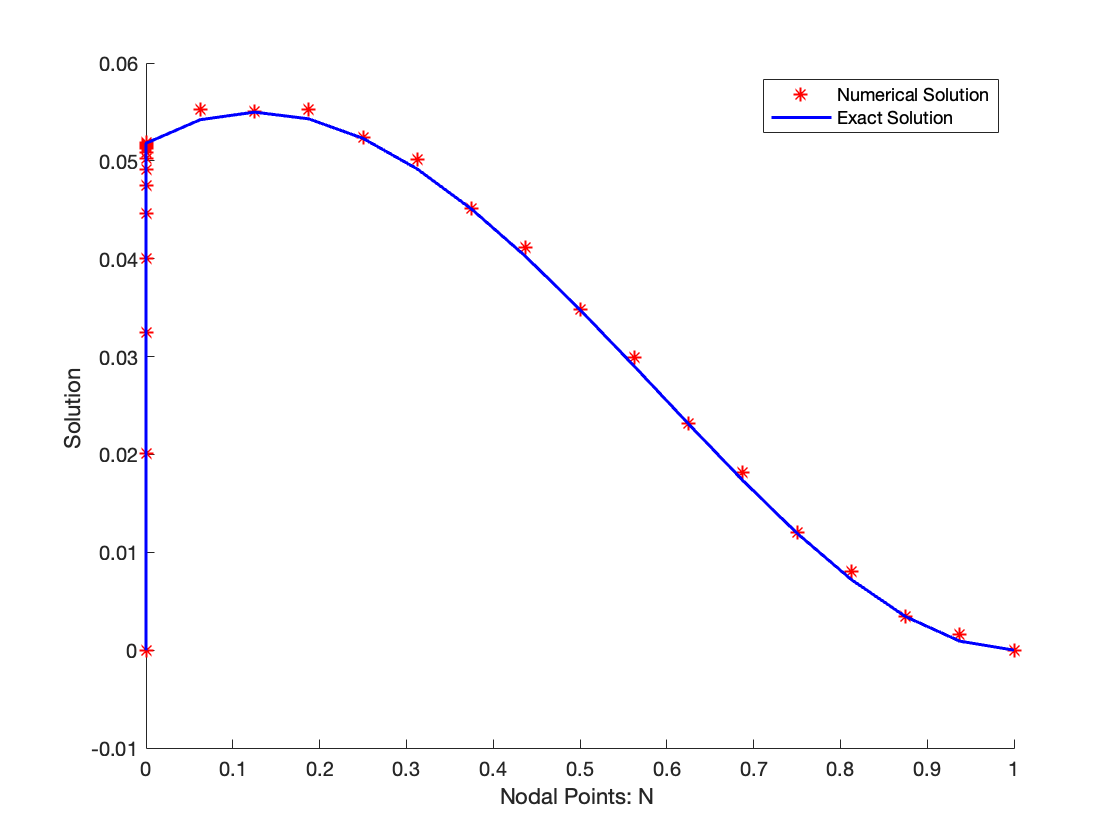}}
\subfigure[]{\includegraphics[width=0.32\textwidth]{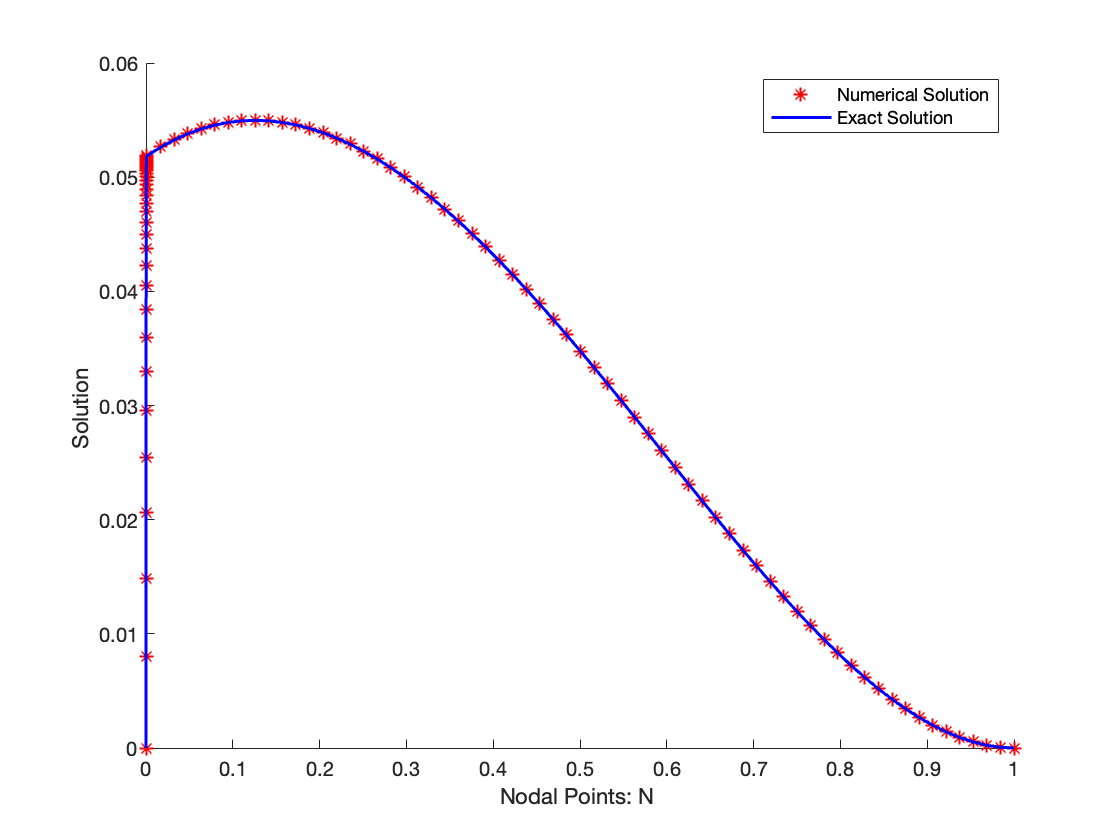}}
\subfigure[]{\includegraphics[width=0.32\textwidth]{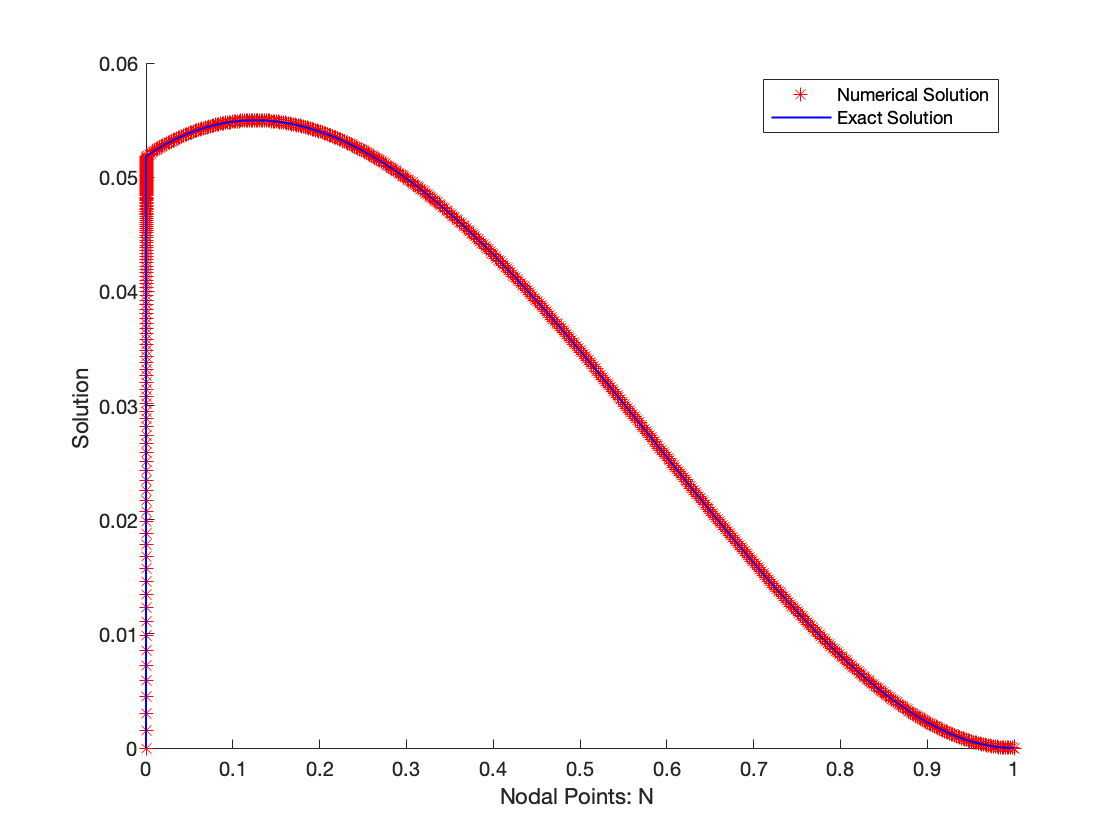}}
\caption{Shishkin mesh:  (a) $N= 32$, $\varepsilon=10^{-8}$; (b) $N= 128$, $\varepsilon=10^{-4}$; (c) $N= 1024$, $\varepsilon=10^{-6}$.}\label{comp2}
\end{figure}

\begin{figure}[!t]
\centering
\subfigure[]{\includegraphics[width=0.32\textwidth]{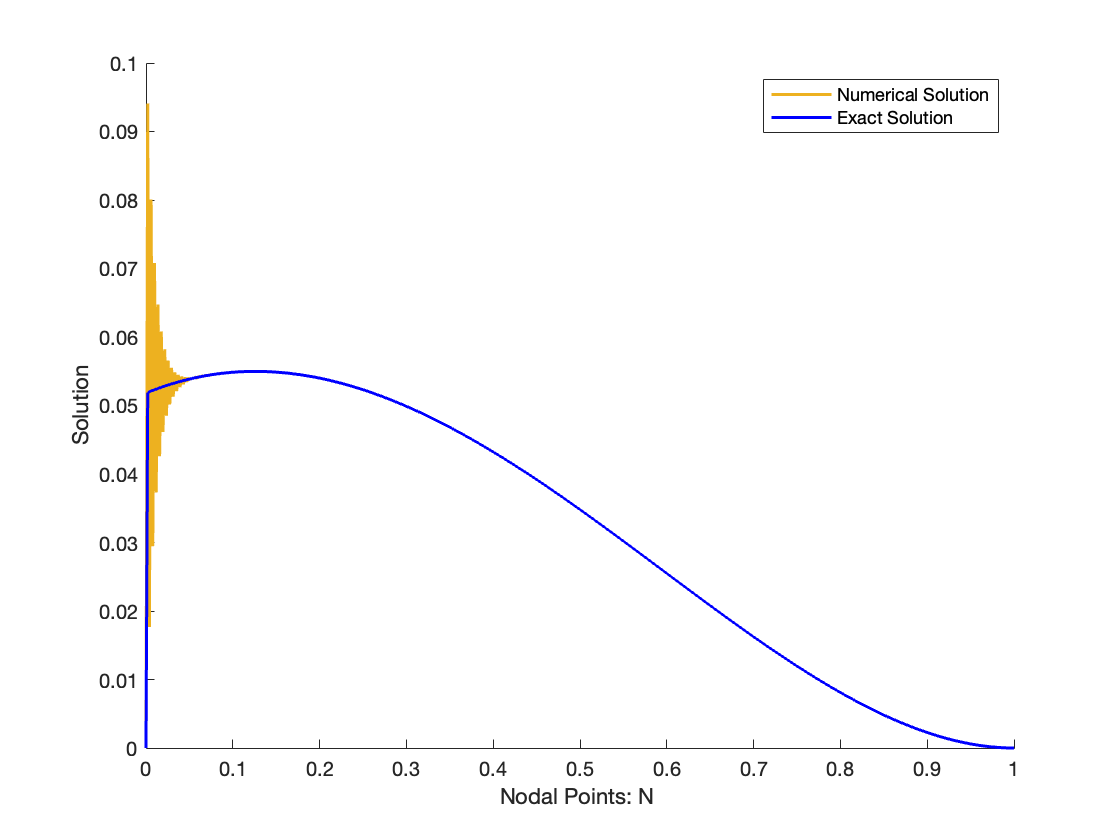}}
\subfigure[]{\includegraphics[width=0.32\textwidth]{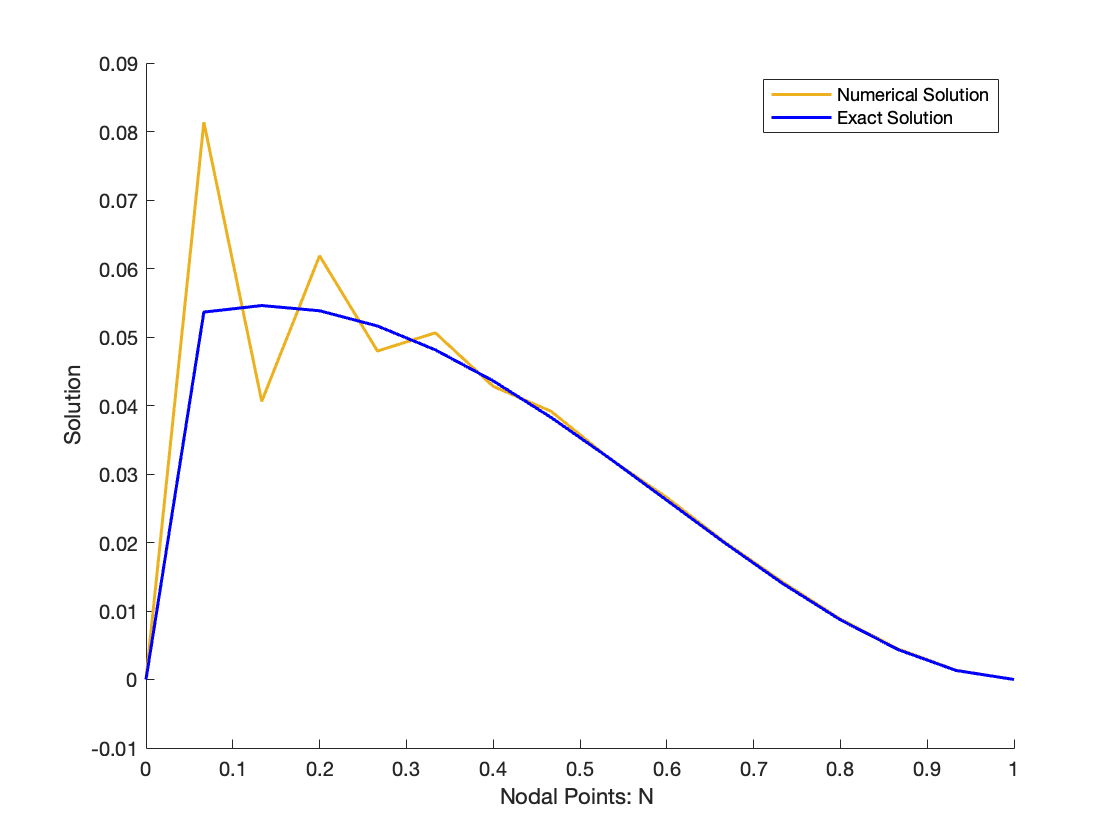}}
\subfigure[]{\includegraphics[width=0.32\textwidth]{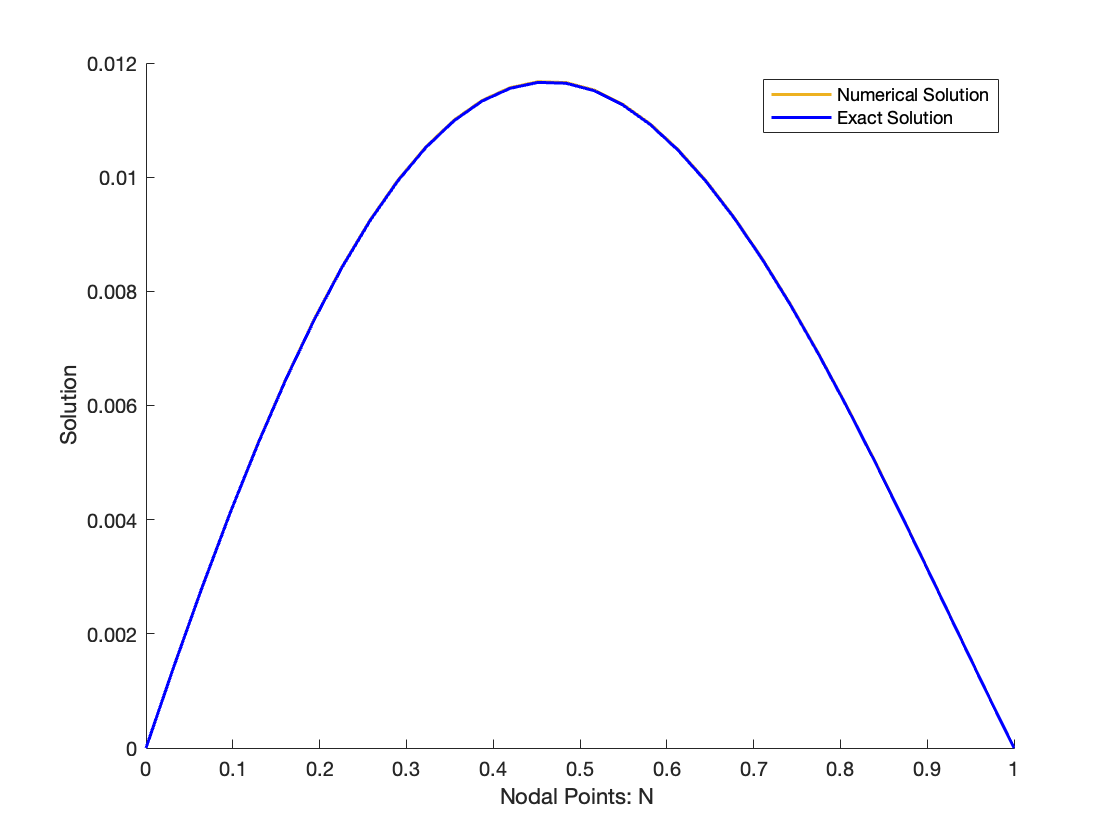}}
\caption{Uniform mesh: (a) $N= 512$, $\varepsilon=10^{-4}$; (b) $N= 16$, $\varepsilon=10^{-2}$; (c) $N= 32$, $\varepsilon= 1 $.}\label{comp2}
\end{figure}

\begin{figure}[!t]
\centering
\subfigure[]{\includegraphics[width=0.32\textwidth]{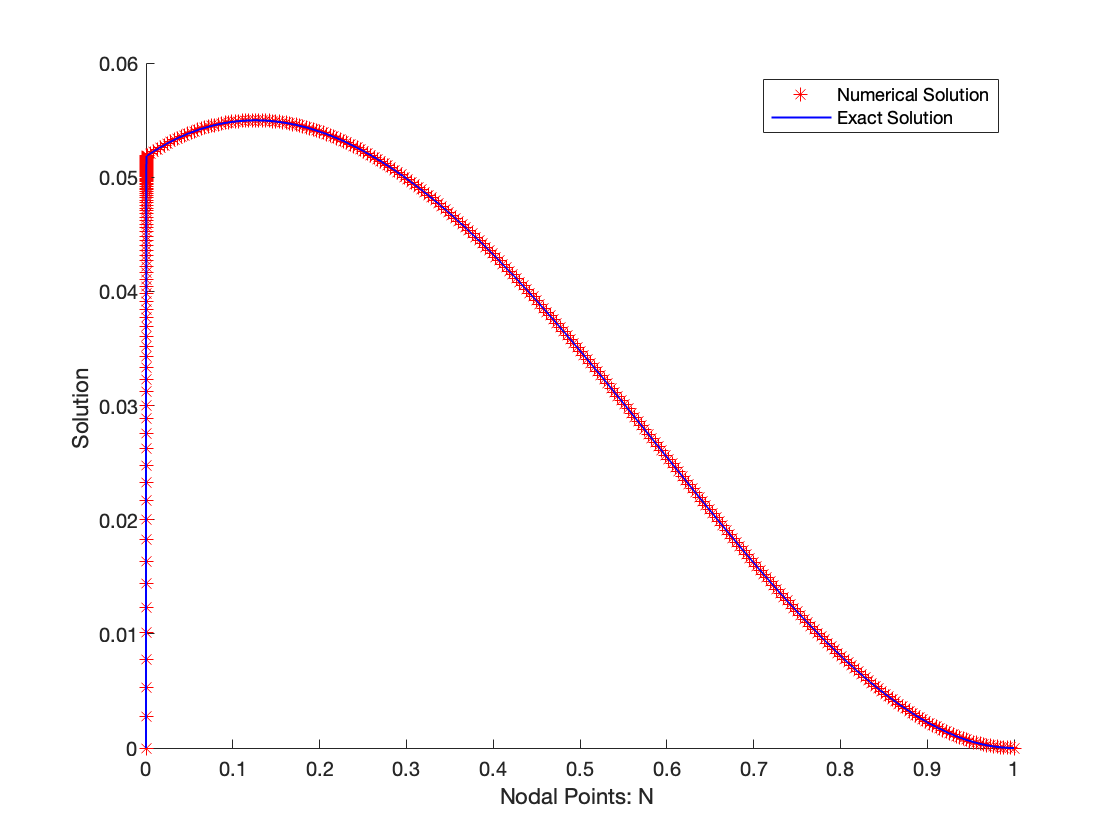}}
\subfigure[]{\includegraphics[width=0.32\textwidth]{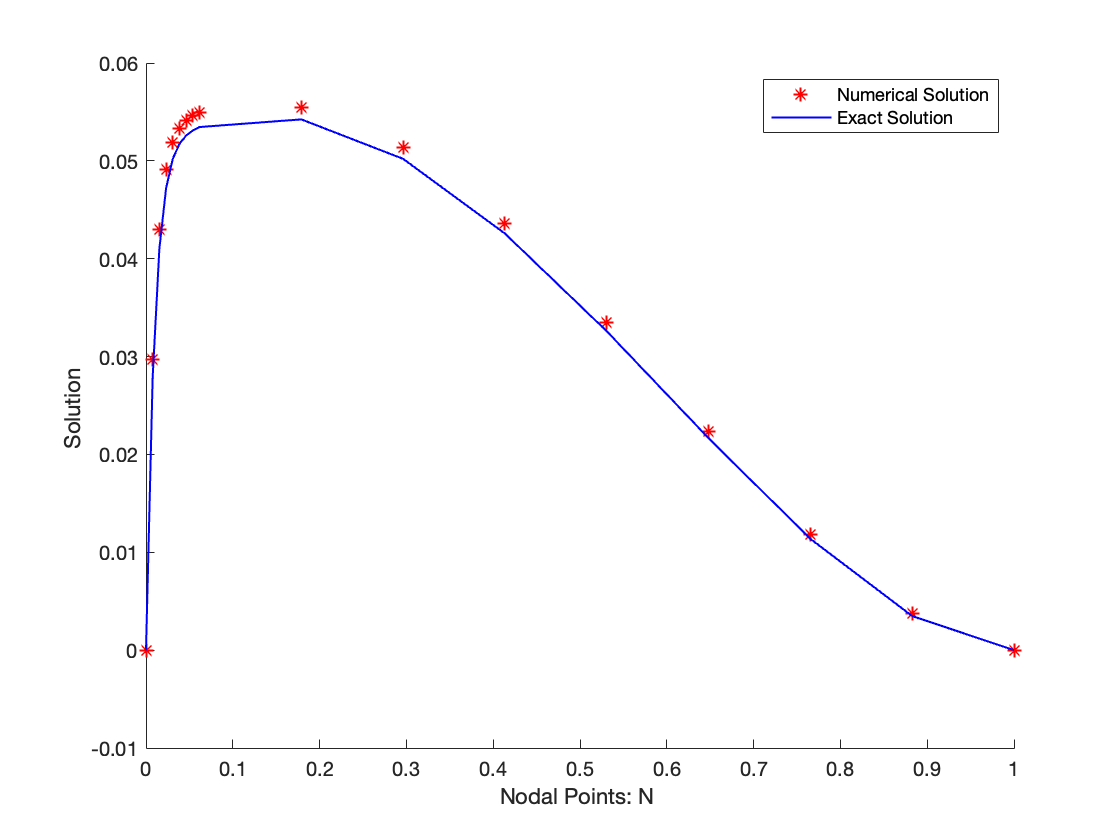}}
\subfigure[]{\includegraphics[width=0.32\textwidth]{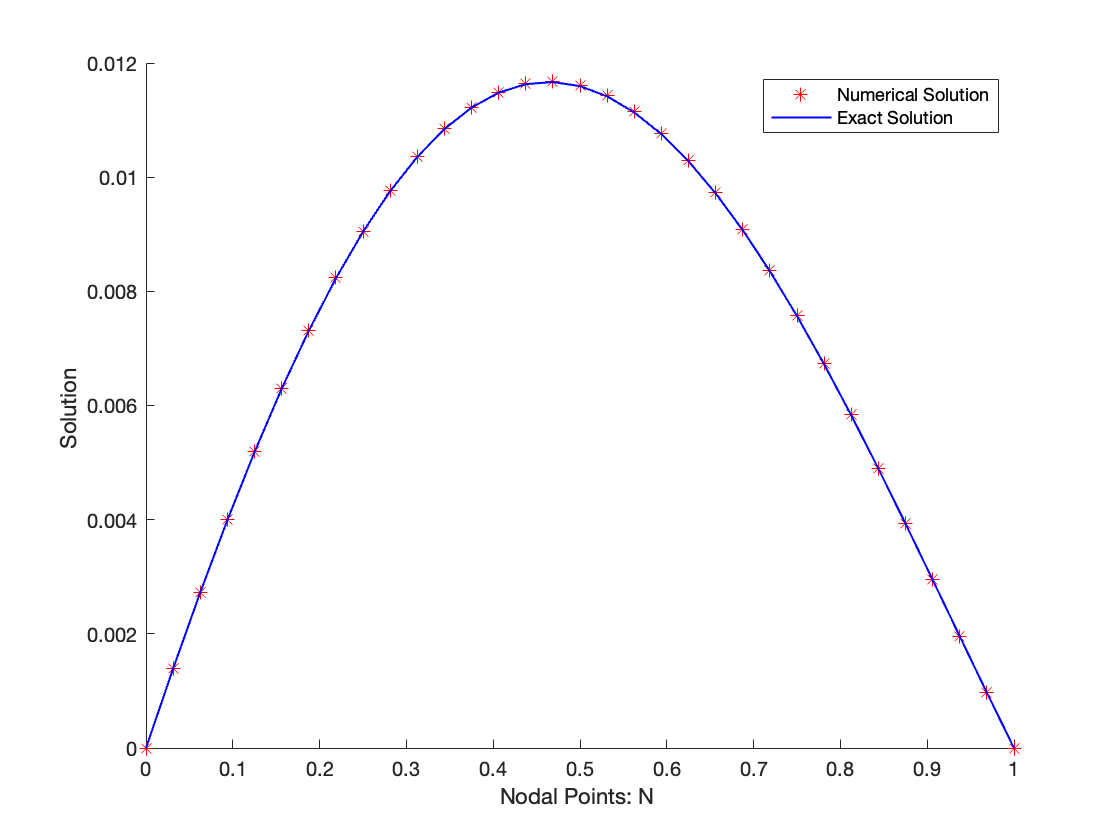}}
\caption{ Shishkin mesh: (a) $N= 512$, $\varepsilon=10^{-4}$; (b) $N= 16$, $\varepsilon=10^{-2}$; (c) $N= 32$, $\varepsilon= 1 $.}\label{comp2}
\end{figure}

\begin{table}[!h] 
\centering
\def\arraystretch{1.2}
\caption{$ \|u-u_n\|_{\infty} $ \mbox{for} $\varepsilon= 10^{-4}, 10^{-2 }$ \mbox{and}  1  $\mbox  {for uniform and Shishkin meshes} $.}
\tabcolsep=10pt
\vspace{-0.5\baselineskip}
\begin{tabular}[c]{|c|c|c|c|c|c|c|}
\hline
& \multicolumn{2}{|c |}{$\epsilon=10^{-4}$}& \multicolumn{2}{|c |}{$\epsilon=10^{-2}$} & \multicolumn{2}{| c|}{$\epsilon=1$}\\
\hline
$N$\multirow{12}{*}{} & {Uniform} & {Shishkin} & {Uniform} & {Shishkin} &{Uniform} &{Shishkin} \\
\cline{2-7}
\hline
 4 & 0.0355& 0.0558& 0.0371& 0.0429 & 0.0011 & 6.9689e-04 \\
\hline
8  & 0.0299 & 0.0152 & 0.0366& 0.0087 &2.2395e-04  & 1.7362e-04\\
\cline{1-7}
\hline
16  & 0.0306 & 0.0039 & 0.0277& 0.0020 & 4.9304e-05 & 4.3367e-05\\
\cline{1-7}
\hline
32  & 0.0350 & 9.8165e-04 & 0.0141 & 0.0010 & 1.1558e-05 & 1.0840e-05\\
\cline{1-7}
\hline
64  & 0.0450& 2.3380e-04 & 0.0046& 5.2399e-04 & 2.7982e-06 & 2.7117e-06\\
\cline{1-7}
\hline
128  & 0.0491 & 5.2171e-05 &   0.0010 & 2.8929e-04 & 6.8864e-07 & 6.7792e-07\\
\cline{1-7}
\hline
256  &   0.0467 & 1.5676e-05 & 2.4386e-04 & 1.5611e-04 & 1.7081e-07 & 1.6948e-07\\
\cline{1-7}
\hline
512  & 0.0422 & 5.3830e-06 & 6.0942e-05 & 8.2720e-05 & 4.2536e-08 & 4.2370e-08\\
\cline{1-7}
\hline
1024  & 0.0342 & 1.8953e-06 & 1.5157e-05 & 4.3196e-05 &1.0613e-08 & 1.0593e-08\\
\cline{1-7}
\hline
2048  & 0.0221 & 7.0754e-07 &3.7826e-06 & 2.2282e-05 &2.6510e-09 & 2.6486e-09\\
\cline{1-7}
\hline
4096  & 0.0097 & 2.8955e-07 & 9.4508e-07 & 1.1371e-05 & 6.6242e-10 & 6.6407e-10\\
\cline{1-7}
\hline
8192 &0.0027 & 1.3217e-07 & 2.3621e-07 & 5.7469e-06 & 1.6553e-10 & 1.7331e-10\\
\cline{1-7}
\hline

\end{tabular}\label{t2}
\end{table}

\end{example}

\begin{example} The purpose of this example is to demonstrate the convergent rates and validate the results in Theorem 5 and Lemma 5. Tables 3 and 4 present the convergent rates of the numerical solutions for different $\varepsilon$ values and mesh sizes $N$. In Table 3, for $\varepsilon=10^{-10}, 10^{-8}, 10^{-6}$, there is no convergent rate under a uniformly refined mesh due to the high singularly perturbed features of the problem. However, as we decrease the singular perturbed properties of the problem, the uniform mesh starts showing convergent rates. For example, in Table 4, when $\varepsilon=1$, the convergent rate $R=2$ is observed under a uniformly refined mesh, which is in strong agreement with Lemma 5. The convergent rates under a Shishkin mesh for $\varepsilon=10^{-10}, 10^{-8}, 10^{-6}, 10^{-4}, 10^{-2}$ are also in strong agreement with Theorem 5. This concludes that our proposed algorithm works well on fourth-order singularly perturbed problems like (\ref{e2}).

Figure 7 supports the same conclusion as the tables. It shows that there is no convergent rate for uniformly refined meshes with high singularly perturbed properties of the problem. However, with a Shishkin mesh, we can observe the expected convergent rates from Theorem 5.

\begin{table}[!h] 
\centering
\def\arraystretch{1.20}
\caption{Convergent rate $ R $ \mbox{ for} $\varepsilon= 10^{-10}, 10^{-8 } \mbox{ and }  10^{-6}  \mbox{ for uniform and Shishkin meshes} $.}
\tabcolsep=10pt
\tabcolsep=10pt
\vspace{-0.5\baselineskip}
\begin{tabular}[c]{|c|c|c|c|c|c|c|}
\hline
& \multicolumn{2}{|c |}{$\epsilon=10^{-10}$}& \multicolumn{2}{|c |}{$\epsilon=10^{-8}$} & \multicolumn{2}{| c|}{$\epsilon=10^{-6}$}\\
\hline
$N$\multirow{10}{*}{} & {Uniform} & {Shishkin} & {Uniform} & {Shishkin} &{Uniform} &{Shishkin} \\
\cline{2-7}

\hline
16  & 0.0210 & 1.9404 & 0.0210 & 1.9404 & 0.0204 & 1.9406\\
\cline{1-7}
\hline
32  & 0.0156 & 1.9737 & 0.0156 & 1.9737 & 0.0179 & 1.9740\\
\cline{1-7}
\hline
64  & 0.0134 & 1.9899 & 0.0135 & 1.9899 & 0.0230 & 1.9909\\
\cline{1-7}
\hline
128  & 0.0080 & 1.9975 &  0.0084 & 1.9976 & 0.0457 & 2.0006\\
\cline{1-7}
\hline
256  &   0.0043 & 2.0008 & 0.0059 & 2.0009 & 0.1372 & 2.0111\\
\cline{1-7}
\hline
512  & 0.0023 & 1.7275 & 0.0086 & 1.7280 & 0.3273 & 1.7333\\
\cline{1-7}
\hline
1024  & 0.0014 & 1.7051 & 0.0260 & 1.7067 & 0.2550 & 1.7325\\
\cline{1-7}
\hline
2048  & 0.0016 &1.7288 & 0.0919  & 1.7340 & 0.0059 & 1.7410\\
\cline{1-7}
\hline
4096  & 0.0044 & 1.7425 & 0.2631 & 1.7596 & 0.0117 & 1.7377\\
\cline{1-7}
\hline
8192 &0.0163 & 1.7330 & 0.3392 &1.7853 & 0.0236 & 1.7522\\
\cline{1-7}
\hline

\end{tabular}\label{rate1}
\end{table}

\begin{table}[!h] 
\centering
\def\arraystretch{1.20}
\caption{Convergent rate $ R $ \mbox{ for} $\varepsilon= 10^{-4}, 10^{-2 } \mbox{ and }  1  \mbox{ for uniform and Shishkin meshes} $.}
\tabcolsep=10pt
\vspace{-0.5\baselineskip}
\begin{tabular}[c]{|c|c|c|c|c|c|c|}
\hline
& \multicolumn{2}{|c |}{$\epsilon=10^{-4}$}& \multicolumn{2}{|c |}{$\epsilon=10^{-2}$} & \multicolumn{2}{| c|}{$\epsilon=1$}\\
\hline
$N$\multirow{10}{*}{} & {Uniform} & {Shishkin} & {Uniform} & {Shishkin} &{Uniform} &{Shishkin} \\
\cline{2-7}

\hline
16  & 0.0297 & 1.9514 & 0.4017 & 2.1152 & 2.1834 & 2.0012\\
\cline{1-7}
\hline
32  & 0.1943 & 2.0032 & 0.9771 & 0.9764 & 2.0927 &2.0003\\
\cline{1-7}
\hline
64  & 0.3657 & 2.0699 & 1.6094 & 0.9658 & 2.0464 & 1.9990\\
\cline{1-7}
\hline
128  & 0.1240 & 2.1640 & 2.1549 & 0.8570 & 2.0227 &2.0000 \\
\cline{1-7}
\hline
256  &   0.0711 & 1.7347 & 2.0872 & 0.8899 & 2.0113 & 2.0000\\
\cline{1-7}
\hline
512  & 0.1478 & 1.5421 & 2.0005 & 0.9163 & 2.0057 & 2.0000\\
\cline{1-7}
\hline
1024  & 0.3023 &1.5060 & 2.0074 & 0.9373 &2.0028 & 2.0000\\
\cline{1-7}
\hline
2048  & 0.6301 &1.4215 &2.0026  & 0.9551 &2.0012 & 1.9998\\
\cline{1-7}
\hline
4096  & 1.1936 & 1.2890 & 2.0009 & 0.9705 & 2.0007 & 1.9958\\
\cline{1-7}
\hline
8192 &1.8130 & 1.1314 & 2.0004 &0.9845 & 2.0006 & 1.9380\\
\cline{1-7}
\hline

\end{tabular}\label{rate2}
\end{table}

\begin{figure}[!t]
\centering
\subfigure[]{\includegraphics[width=0.49\textwidth]{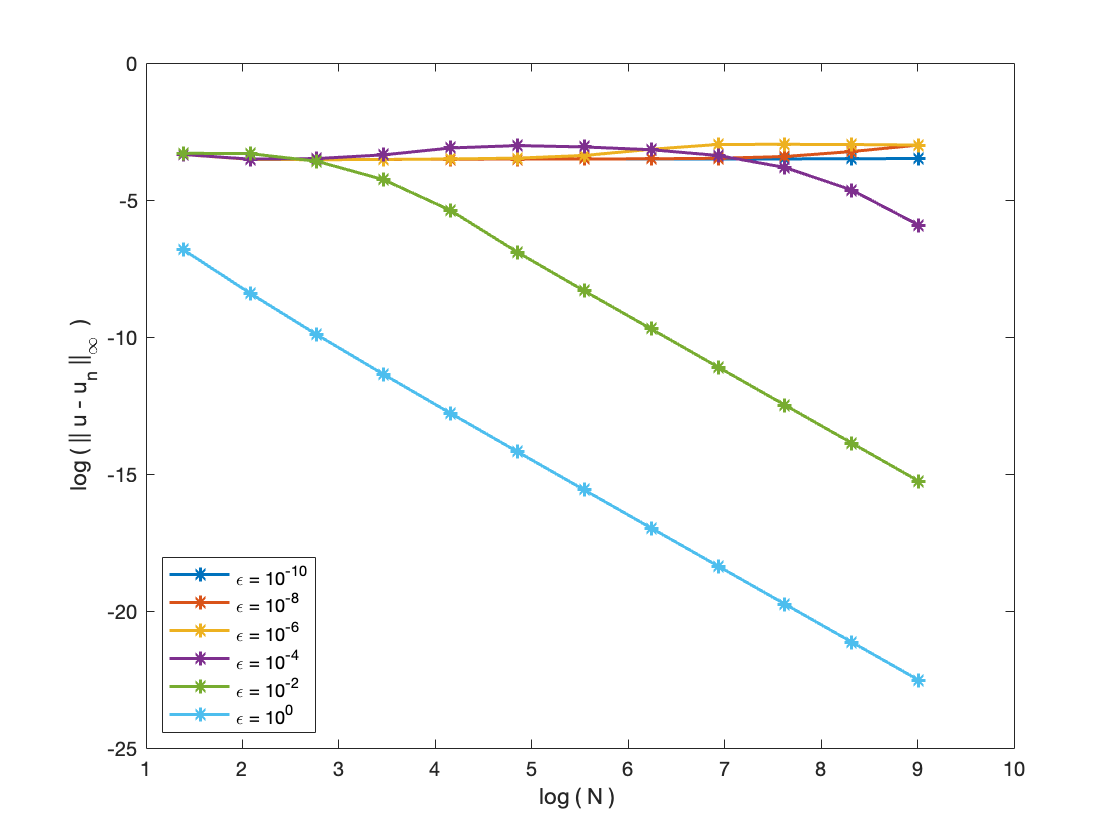}}
\subfigure[]{\includegraphics[width=0.49\textwidth]{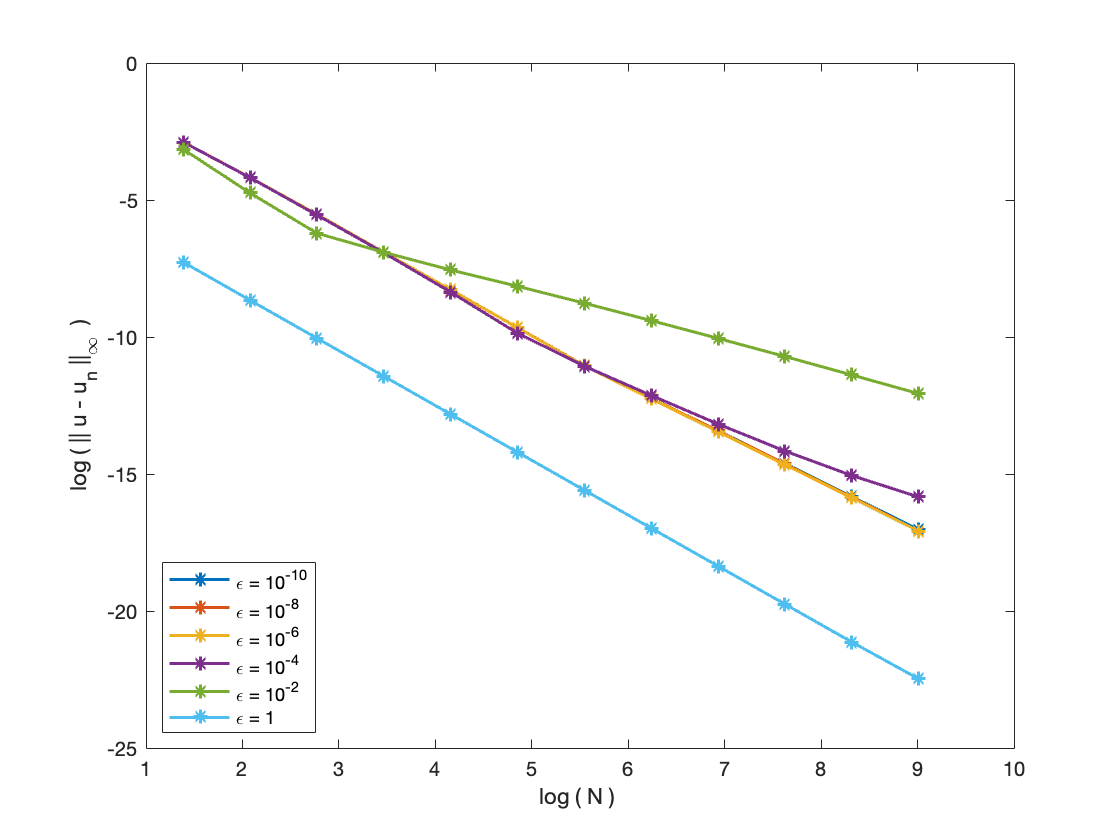}}

\caption{(a) Convergent rate plots under uniform mesh; (b) Convergent rate plots under Shishkin Rate .}\label{ratecomp}
\end{figure}

\end{example}

\begin{example} In this example, we are comparing the CPU time required by our finite element algorithm to solve problem (\ref{e2}) on uniform and Shishkin meshes. The results are presented in Table 5 and Table 6. From the tables, we can observe that the algorithm converges to the true solution faster under the Shishkin mesh compared to the uniform mesh.

\begin{table}[!h] 
\centering
\def\arraystretch{1.20}
\caption{The $L^\infty$ error $\|u_R^j-\phi_j\|_{L^\infty(\Omega)}$ in the L-shaped domain  on quasi-uniform meshes.}
\tabcolsep=10pt
\vspace{-0.5\baselineskip}
\begin{tabular}[c]{|c|c|c|c|c|c|c|}
\hline
& \multicolumn{2}{|c |}{$\epsilon=10^{-10}$}& \multicolumn{2}{|c |}{$\epsilon=10^{-8}$} & \multicolumn{2}{| c|}{$\epsilon=10^{-6}$}\\
\hline
$N$\multirow{6}{*}{} & {Uniform} & {Shishkin} & {Uniform} & {Shishkin} &{Uniform} &{Shishkin} \\
\cline{2-7}
\hline
512  & 0.05& 0.09 & 0.08 &0.08 & 0.05 & 0.07\\
\cline{1-7}
\hline
1024  & 0.13 &0.13 &0.15 & 0.11 &0.11 & 0.11\\
\cline{1-7}
\hline
2048  & 0.43 &0.26 &0.41  & 0.26 &0.40 & 0.27\\
\cline{1-7}
\hline
4096  & 1.43 & 0.76 & 1.42 & 0.78 & 1.24 & 0.81 \\
\cline{1-7}
\hline
8192 &5.21 & 2.64 & 4.90 &2.88 & 4.12 & 2.92\\
\cline{1-7}
\hline
16384  & 20.06 & 13.27 &18.69 & 14.05 & 17.69 & 13.52\\
\cline{1-7}

\end{tabular}\label{cputime1}
\end{table}

\begin{table}[!h] 
\centering
\def\arraystretch{1.20}
\caption{The $L^\infty$ error $\|u_R^j-\phi_j\|_{L^\infty(\Omega)}$ in the L-shaped domain  on quasi-uniform meshes.}
\tabcolsep=10pt
\vspace{-0.5\baselineskip}
\begin{tabular}[c]{|c|c|c|c|c|c|c|}
\hline
& \multicolumn{2}{|c |}{$\epsilon=10^{-4}$}& \multicolumn{2}{|c |}{$\epsilon=10^{-2}$} & \multicolumn{2}{| c|}{$\epsilon=1$}\\
\hline
$N$\multirow{6}{*}{} & {Uniform} & {Shishkin} & {Uniform} & {Shishkin} &{Uniform} &{Shishkin} \\
\cline{2-7}
\hline
512  & 0.06 & 0.07 & 0.05 & 0.09 &0.01 & 0.07\\
\cline{1-7}
\hline
1024  & 0.14 & 0.11 & 0.11 & 0.12 &0.05 & 0.11 \\
\cline{1-7}
\hline
2048  & 0.37 & 0.24 & 0.34  & 0.27 &0.22 & 0.27 \\
\cline{1-7}
\hline
4096  & 1.15 & 0.72 & 1.13 & 0.78 & 0.91 & 0.85\\
\cline{1-7}
\hline
8192 & 4.00 & 2.75 & 3.86 & 2.73 & 3.78 & 2.88\\
\cline{1-7}
\hline
16384  & 17.34 & 14.11 & 17.27 & 13.04 & 17.00 & 14.23\\
\cline{1-7}

\end{tabular}\label{cputime2}
\end{table}

\end{example}

\section{Conclusion}

It is worth noting that the proposed method can be applied to a wide range of singularly perturbed problems, not only the specific type of problem studied in this work. The decoupling technique presented in this paper can be adapted to other types of singularly perturbed problems, and the numerical scheme can be modified accordingly. Additionally, this work serves as a foundation for future research in the field of singularly perturbed problems, and can be used as a benchmark for comparing the performance of other numerical methods. Ultimately, we expect that it may be possible apply the proposed method for 2D with a completely new mesh  algorithm as an alternative to a standard Shishkin mesh as shown in the figure 6. This is the subject of an ongoing work. Overall, the proposed method provides a promising approach for solving singularly perturbed problems efficiently and accurately.

\begin{figure}[!t]
\centering
\includegraphics[width=1.0\textwidth]{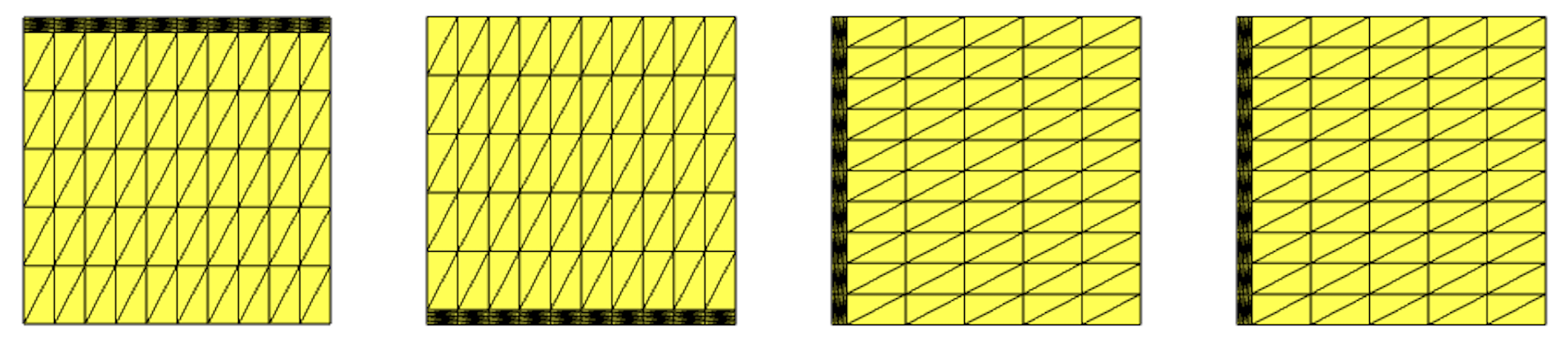}
\caption{Standard 2D Shishkin mesh with boundary layers on top, bottum, left and right.}\label{shi}
\end{figure}

\section*{Acknowledgments}

The author thank the anonymous reviewers for their valuable time.

\section*{Funding}
This research is supported by postdoctoral scholar program in the Department of Mathematics at the University of Kentucky in Lexington, Kentucky.

\end{document}